\documentclass[11pt]{amsart}
\textheight
9in
\usepackage{amsmath,amsthm,amsfonts,amssymb,latexsym,epsfig}

\newtheorem{theorem}{Theorem}[section]

\newtheorem{lemma}{Lemma}[section]

\newtheorem{cor}{Corollary}[section]

\numberwithin{equation}{section}

\theoremstyle{definition}

\theoremstyle{remark}

\begin{document}
\title{Some completely monotonic functions involving the $q$-gamma function}
\author{Peng Gao}
\address{Division of Mathematical Sciences, School of Physical and Mathematical Sciences,
Nanyang Technological University, 637371 Singapore}
\email{penggao@ntu.edu.sg}

\begin{abstract}
We present some completely monotonic functions involving the
$q$-gamma function that are inspired by their analogues involving
the gamma function.
\end{abstract}

\maketitle
\section{Introduction}
\label{sec 1} \setcounter{equation}{0}
   The $q$-gamma function is defined for a complex number $z$ and $q \neq 1$ by
\begin{eqnarray}
\label{01}
  \Gamma_q(z)&=& \left\{\begin{array}{ll} \frac {(q;q)_{\infty}}{(q^z;q)_{\infty}}
      (1-q)^{1-z}, & 0<q<1; \\
      \frac
{(q^{-1};q^{-1})_{\infty}}{(q^{-z};q^{-1})_{\infty}}(q-1)^{1-z}q^{\frac
1{2}z(z-1)}, & q>1.
\end{array}\right.
\end{eqnarray}
    where the product $(a;q)_{\infty}$ is defined by
\begin{equation*}
   (a;q)_{\infty}=\prod^{\infty}_{n=0}(1-aq^n).
\end{equation*}

    In what follows we restrict our attention to positive real numbers $x$. We note here \cite{Koo} the limit of $\Gamma_q(x)$ as $q \rightarrow
    1^{-}$ gives back the well-known Euler's gamma function:
\begin{equation*}
    \lim_{q \rightarrow
    1^{-}}\Gamma_q(x)=\Gamma(x)=\int^{\infty}_0 t^x e^{-t}\frac
    {dt}{t}.
\end{equation*}
    It's then easy to see using \eqref{01} that $\lim_{q \rightarrow 1}\Gamma_q(x)=\Gamma(x)$. For
historical remarks on gamma and $q$-gamma functions, we
    refer the reader to \cite{Koo}, \cite{alz1} and \cite{alz1.5}.

    There exists an extensive and rich literature on inequalities for the gamma and $q$-gamma functions of positive real numbers.
For the recent developments in this area, we refer the reader to
the articles \cite{I&M}, \cite{alz1}-\cite{alz2}, \cite{Q&V} and
the references therein. Many of these inequalities follow from the
monotonicity properties of functions which are closely related to
$\Gamma$ (resp. $\Gamma_q$) and its logarithmic derivative $\psi$
(resp. $\psi_q$) as $\psi'$ and $\psi'_q$ are completely
    monotonic functions on $(0, +\infty)$ (see \cite{K&M}, \cite{alz2}). Here we recall that a function $f(x)$ is said to
be completely monotonic on $(a, b)$ if it has derivatives of all
orders and $(-1)^kf^{(k)}(x) \geq 0, x \in (a, b), k \geq 0$. We
further note that Lemma 2.1 of \cite{B&I} asserts that
$f(x)=e^{-h(x)}$ is completely monotonic on an interval if $h'$
is. Following \cite{G&I}, we call such functions $f(x)$
logarithmically completely monotonic.

  We note here that $\lim_{q \rightarrow 1}\psi_q(x)=\psi(x)$ (see
     \cite{K&S}), hence in what follows we also write $\Gamma_1(x)$ for $\Gamma(x)$ and $\psi_1(x)$ for $\psi_q(x)$.
     Thus we may also regard the gamma function as a $q$-gamma function with $q=1$
     and in this manner, many completely monotonic functions involving $\Gamma_q(x)$ and $\psi_q(x)$
      are inspired by their analogues involving $\Gamma(x)$ and $\psi(x)$.
      It is our goal in this paper to present some completely monotonic functions involving $\Gamma_q, \psi_q$ that are motivated
      by this point of view. In the remaining part of this introduction, we briefly mention the motivations for our results in the paper.

     In \cite{Ker}, Kershaw proved the $q=1$ case of the following result for
     $0<s<1, x>0$:
\begin{align}
\label{02}
    e^{(1-s)\psi_q(x+s^{1/2})} < \frac {\Gamma_q(x+1)}{\Gamma_q(x+s)} <
    e^{(1-s)\psi_q(x+(s+1)/2)}.
\end{align}
    A result of Ismail and
   Muldoon \cite{I&M} establishes the second inequality in \eqref{02} for $0<q<1$. In \cite{B&I}, Bustoz and Ismail showed that when $q=1$,
   the function
   ($0<s<1$)
\begin{align*}
   x \mapsto \frac
   {\Gamma_q(x+s)}{\Gamma_q(x+1)}e^{(1-s)\psi_q(x+(s+1)/2)}
\end{align*}
    is completely monotonic on $(0, +\infty)$. In \cite{Gao1}, it
    is shown that the result of Bustoz and Ismail also holds for
    any $q>0$.

   In \cite{alz1.5}, Alzer asked to determine the best possible values of $a(q,s)$ and $b(q,s)$ such that the
     following inequalities hold for all $x>0, 0<q \neq 1, 0<s<1$:
\begin{align}
\label{03}
    e^{(1-s)\psi_q(x+a(q,s))}<\frac
    {\Gamma_q(x+1)}{\Gamma_q(x+s)}<e^{(1-s)\psi_q(x+b(q,s))}.
\end{align}

    We shall determine the best possible values of $a(q,s)$ and $b(q,s)$ in Section \ref{sec 4}.
    Another result given in Section \ref{sec 4} is motivated by the following
    result of Alzer and Batir \cite{AB}, who showed that the
function ($x > 0, c \geq 0$)
\begin{equation*}
   G_{c}(x)=\ln \Gamma(x)-x\ln x+x-\frac 1{2} \ln (2\pi)+\frac 1{2}\psi(x+c)
\end{equation*}
   is completely monotonic if and only if $c \geq 1/3$ and $-G_c(x)$ is completely monotonic if and only if $c = 0$.
   We shall present a $q$-analogue in Section \ref{sec 4} for
$G'_c(x)$.

    Muldoon \cite{M} studied the
monotonicity property of the function
\begin{equation*}
   h_{\alpha}(x)=x^{\alpha}\Gamma(x)(e/x)^x.
\end{equation*}
   He showed that $h_{\alpha}(x)$ is logarithmically completely monotonic on $(0, +\infty)$ for $\alpha \leq 1/2$.
   We point out here that as was shown in \cite[Theorem 3.3]{A&B},
   $1/2$ is the largest possible number to make the assertion hold for $h_{\alpha}(x)$. In \cite[Proposition  4.1]{Gao1},
   it is shown that if one defines for $\alpha \geq 0$,
\begin{equation}
\label{4.1}
  f_{\alpha}(x)=-\ln \Gamma(x)+(x-\frac 1{2})\ln x -x +\frac
  {1}{12}\psi'(x+\alpha),
\end{equation}
  then $f'_{\alpha}(x)$ is completely monotonic on $(0, +\infty)$ if  $\alpha \geq 1/2$ and $-f'_{\alpha}(x)$
  is completely monotonic on $(0, +\infty)$ if $\alpha =0$.
  As was pointed out in \cite{Gao1}, this implies a result of Alzer \cite[Theorem 1]{alz0}.
  In Section \ref{sec 4}, we shall establish a $q$-analogue of the above result.

  It's shown in the proof of Theorem 2.2 in \cite{E&G&P1} that for $x>0$ and $0<q<1$,
\begin{equation}
\label{1.1}
  \psi'_q(x+1) < \frac {\ln (1/q)q^x}{1-q^x}.
\end{equation}
   The $q =1$ analogue of inequality \eqref{1.1} is $\psi'(x+1) \leq 1/x$, which reminds us the following asymptotic expansion \cite[(1.5)]{alz2} for the derivatives of $\psi(x)$:
\begin{equation}
\label{1.2}
  (-1)^{n+1}\psi^{(n)}(x) = \frac
  {(n-1)!}{x^n}+\frac{n!}{2x^{n+1}}+O\left(\frac {1}{x^{n+2}} \right ), \ n \geq 1, \ x
\rightarrow +\infty.
\end{equation}
   We note that Lemma 2.2 of \cite{Gao} asserts that for fixed $n \geq 1, a \geq 0$, the function
$f_{a,n}(x)= x^n(-1)^{n+1}\psi^{(n)}(x+a)$ is increasing on $[0,
+\infty)$
  if and only if $a \geq 1/2$. It follows from this and \eqref{1.2} that we have $\psi'(x+1/2) \leq 1/x$
  and this suggests that inequality \eqref{1.1} would still hold if one replaces $\psi'_q(x+1)$ with
  $\psi'_q(x+1/2)$. We shall show this is indeed the case in
  Section \ref{sec 4}.

\section{Lemmas}
\label{sec 3} \setcounter{equation}{0}
   The following lemma gathers a few results on
   $\Gamma_q$ and $\psi_q$. Equality \eqref{1.12} below is given in \cite[(2.7)]{alz1.5} and the rest can be easily derived
   from \eqref{01} and \eqref{1.12}.
\begin{lemma}
\label{lem1} For $0<q<1$, $x>0$,
\begin{align}
\label{1.12}
  \psi_q(x)  &=   -\ln (1-q) + \ln q \sum^{\infty}_{n=1}\frac {q^{nx}}{1-q^n},  \\
\label{2.2}
   \ln \Gamma_q(x+1) & =\ln \Gamma_q(x)+\ln \frac {1-q^x}{1-q}, \\
\label{2.7}
    \psi_q(x+1)&=\psi_q(x)-\frac {(\ln q)q^x}{1-q^x}, \\
\label{2.7'}
    \psi'_q(x+1) &=\psi'_q(x)-\frac {(\ln
    q)^2q^x}{(1-q^x)^2}.
\end{align}
\end{lemma}

   Our next lemma is a result in \cite{S}:
\begin{lemma}
\label{lem1'} For positive numbers $x \neq y$ and real number $r$,
we define
\begin{equation*}
  E(r,0;x,y)=\left ( \frac {1}{r} \cdot \frac {x^r-y^r}{\ln x-\ln y }\right
  )^{1/r}, \, r\neq 0; \, E(0,0;x,y)=\sqrt{xy}.
\end{equation*}
  Then the function $r \mapsto E(r,0;x,y)$ is strictly increasing on $\mathbb{R}$.
\end{lemma}

\begin{lemma}
\label{lem2}
   Let $0<q<1$, then for any integer $n \geq 1$,
\begin{eqnarray}
\label{2.5}
&& \frac {\ln q}{1-q^n}+\frac 1{n}-\frac {\ln q}{2}-\frac {n^2(\ln q)^3q^{n/2}}{12(1-q^n)} < 0, \\
\label{2.6} && \frac {\ln q}{1-q^n}+\frac 1{n}-\frac {\ln
q}{2}-\frac {n^2(\ln q)^3}{12(1-q^n)} > 0.
\end{eqnarray}
\end{lemma}
\begin{proof}
  On setting $\ln q^n =x$, it is easy to see that inequality \eqref{2.5} follows from $f(x) < 0$ for $x<0$, where
\begin{equation*}
  f(x)=6x(1+e^x)+12(1-e^x)-x^3e^{x/2}.
\end{equation*}
  As $f''(x)=6xe^{x/2}(e^{x/2}-1-x/2-x^2/24)$ and it is easy to see that there is a unique solution $x_0 \in (-\infty, 0)$
  of the equation $e^{x/2}-1-x/2-x^2/24=0$, it follows that
  $f''(x)>0$ for $x<x_0$ and $f''(x)<0$ for $x_0<x<0$. One then
  deduces easily via the expression of $f'(x)$ and the observation $f'(0)=0$ that $f'(x)>0$ for $x<0$. It follows from this and $f(0)=0$
  that $f(x)<0$ for $x<0$.

  Similarly, inequality \eqref{2.6} follows from $g(x) > 0$ for $x<0$, where
\begin{equation*}
  g(x)=6x(1+e^x)+12(1-e^x)-x^3.
\end{equation*}
  As $g''(x)=6x(e^x-1)>0$ for $x<0$ and $g'(0)=0$, we see that $g'(x)<0$ for $x<0$ and it follows from this and $g(0)=0$ that $g(x)>0$
  for $x<0$ and this completes the proof.
\end{proof}
\section{Main Results}
\label{sec 4} \setcounter{equation}{0}

    We first determine the best possible value for $a(q,s)$ in
    \eqref{03}. For this, for any $q>0, t>s>0$, we denote
    $I_{\psi_q}(s,t)$ as
    the integral $\psi_q$ mean of $s$ and $t$:
\begin{align}
\label{3.01}
   I_{\psi_q}(s,t)=\psi^{-1}_q\left ( \frac {1}{t-s}\int^{t}_{s}\psi_q(u)du
\right ).
\end{align}
    Then we have the following result:
\begin{theorem}
\label{thm4'} For every $q>0, x > 0, t>s>0$, we have
\begin{align*}
   \psi_q \left(x+I_{\psi_q}(s,t) \right) < \frac
   {1}{t-s}\int^{t}_{s}\psi_q(x+u)du,
\end{align*}
   where the constant $I_{\psi_q}(s,t)$ is best possible.
\end{theorem}
\begin{proof}
   We note that the case $q=1$ of the assertion of the theorem is
   already established in \cite[Thereom 4]{E&G&P}. The general
   case can be established similarly, on noting that the function
\begin{align*}
   x \mapsto I_{\psi_q}(x+s,x+t)-x
\end{align*}
    is increasing by Theorem 4 of \cite{E&P}, in view that
    $\psi'_q$ is completely monotonic on $(0, +\infty)$. On
    considering the case $x \rightarrow 0^{+}$, we see immediately
    that the constant $I_{\psi_q}(s,t)$ is best possible and this
    completes the proof.
\end{proof}

    On setting $t=1$ in Theorem \ref{thm4'}, we readily deduce the
    following result concerning the best possible value $a(q,s)$
    in \eqref{03}:
\begin{cor}
\label{cor2.0} Let $q > 0$ and $0<s<1$. The first inequality of
\eqref{03} holds for all $x>0$ with the best possible value
$a(q,s)=I_{\psi_q}(s,1)$, where $I_{\psi_q}$ is defined as in
\eqref{3.01}.
\end{cor}
     Now to determine the best possible value for $b(q,s)$ in
     \eqref{03}, we note that it is easy to see on considering the case $x \rightarrow +\infty$
that the best possible value for $b(q,s)$ is $(1+s)/2$ when $q>1$.
     When $0<q<1$, we have the following result:
\begin{theorem}
\label{thm4} Let $0<q<1$ and $0<s<1$. Let
\begin{align*}
    b(q,s)=\frac {\ln \frac {q^s-q}{(s-1)\ln q}}{\ln q}.
\end{align*}
    For $x>0$, let
\begin{align*}
    f_{q,s,c}(x)=\ln \Gamma_q(x+1)-\ln
    \Gamma_q(x+s)-(1-s)\psi_q(x+c),
\end{align*}
    where $c>0$. Then $-f_{q,s,c}(x)$ is completely monotonic on $(0, +\infty)$ if and only if $c \geq b(q,s)$.
\end{theorem}
\begin{proof}
   We have, using \eqref{1.12}, that
\begin{align*}
  f'_{q,s, b(q,s)}(x) & =\psi_q(x+1)-\psi_q(x+s)-(1-s)\psi'_q(x+b(q,s))  \\
  &=\ln
  q\sum^{\infty}_{n=1}\frac {q^{nx}}{1-q^n}\left (q^n-q^{ns}-(1-s)(\ln q^n)q^{nb(q,s)} \right
  ).
\end{align*}
   We want to show $q^n-q^{ns}-(1-s)(\ln q^n)q^{nb(q,s)} \leq 0$,
   which is equivalent to $E^{s-1}(n(s-1), 0; q, 1) \geq
   q^{b(q,s)-1}$, where $E$ is defined as in Lemma \ref{lem1'}.
   It also follows from Lemma \ref{lem1'} that $E^{s-1}(n(s-1), 0; q, 1)
   \geq E^{s-1}(s-1, 0; q, 1)=q^{b(q,s)-1}$. We then deduce that
   $f'_{q,s,c}(x)$ is completely monotonic on $(0, +\infty)$ when $c \geq b(q,s)$. This
   together with the observation that $\lim_{x \rightarrow
   +\infty}f_{q,s,c}(x)=0$ implies the ``if" part of the assertion of the theorem.

   To show the ``only if" part of the assertion of the theorem, we
   use \eqref{2.2} and \eqref{2.7} to deduce that
\begin{align*}
   f_{q,s,c}(x+1)-f_{q,s,c}(x)=\ln \frac
   {1-q^{x+1}}{1-q^{x+s}}+(1-s)\ln q\frac {q^{x+c}}{1-q^{x+c}}.
\end{align*}
   If we set $z=q^x$ and consider the Taylor expansion of the
    above expression at $z=0$, then the first order term is:
 \begin{align*}
    \left(q^s-q+\left(1-s \right )\left(\ln q \right )q^c \right )z.
 \end{align*}
    Note that the expression in the parenthesis above is $<0$ if
    $c<b(q,s)$ as it is $0$ when $c=b(q,s)$. This implies that $f_{q,s,c}(x+1)<f_{q,s,c}(x)$ when $x$
    is large enough and this shows that $-f_{q,s,c}(x)$ can't be
    completely monotonic on $(0, +\infty)$ when $c<b(q,s)$ and this
    completes the proof of the ``only if" part of the assertion of
    the theorem.
\end{proof}

   Theorem \ref{thm4} now allows us to determine the best possible value of
   $b(q,s)$ in \eqref{03} when $0<q<1$ in the following:
\begin{cor}
\label{cor2.1} Let $0<q<1$ and $0<s<1$. The inequality
\begin{align}
\label{2.8}
   \frac
    {\Gamma_q(x+1)}{\Gamma_q(x+s)}< e^{(1-s)\psi_q(x+b(q,s))}
\end{align}
   holds for all $x>0$ with the best possible value $b(q,s)$ given as in the statement of Theorem
   \ref{thm4}.
\end{cor}
\begin{proof}
   Using the same notions in the proof of Theorem \ref{thm4}, we see from the proof of Theorem \ref{thm4} that
   $f'_{q,s,b(q,s)}(x)>0$ for $x>0$, which implies the strict inequality in
   \eqref{2.8}. To show $b(q,s)$ is best possible, we note that
   in the proof of Theorem \ref{thm4}, we've shown that
   $f_{q,s,c}(x+1)-f_{q,s,c}(x)<0$ for $x$ large enough if
   $c<b(q,s)$. It follows that $f_{q,s,c}(x+k)-f_{q,s,c}(x)<0$ for
   any positive integer $k$ when $x$ is large enough and
   $c<b(q,s)$. On letting $k \rightarrow +\infty$, we see
   immediately that this implies that $-f_{q,s,c}(x)<0$, so that
   inequality \eqref{2.8} fails to hold with $b(q,s)$ being
   replaced by any $c<b(q,s)$ and this completes the proof.
\end{proof}

   We note here that Corollary \ref{cor2.1} refines a
   result of Ismail and Muldoon in \cite{I&M}, mentioned in the introduction of this paper, where $b(q,s)$ is replaced by
$(1+s)/2$ in \eqref{2.8}. One can also check directly that $b(q,s)
\leq (1+s)/2$, as it follows from $E(s-1, 0; q, 1) \leq E(0,0;
q,1)$.
   Moreover, it is easy to see that when $q
   \rightarrow 1^{-}$, $b(q,s) \rightarrow (1+s)/2$ and in this
   case \eqref{2.8} gives back the second inequality in \eqref{02} for $q=1$.


   Our next result is a $q$-analogue of the result of Alzer and and Batir \cite{AB} mentioned in Section \ref{sec 1}.
\begin{theorem}
 \label{thm3} Let $0<q<1$ be fixed. Let $c \geq 0$. Let $a_q=(q-1-\ln q)/(\ln q)^2$. The function
 \begin{align*}
    g_{q,c}(x)=\psi_q(x)-\ln\frac {1-q^x}{1-q}+a_q\psi'_q(x+c)
 \end{align*}
    is completely monotonic on $(0, +\infty)$ if and only if $c =0 $.
\end{theorem}
\begin{proof}
    We have, using \eqref{1.12}, that
\begin{align*}
    g_{q,c}(x)= \ln q \sum^{\infty}_{n=1}\frac {q^{nx}}{1-q^n}\left (1+\frac {1-q^n}{n\ln q}+a_q(\ln q^n)q^{nc} \right
   ).
 \end{align*}
    On setting $t=-\ln q^n$, we have $t \geq -\ln q$ and the expression in
    the parenthesis above when $c=0$ can be rewritten as
 \begin{align*}
    1-\frac {1-e^{-t}}{t}-a_qt=\frac {1}{t}(-1+t+e^{-t}-a_qt^2):=h_q(t)/t.
 \end{align*}
     It suffices to show $h_q(t) \leq 0$ for $t \geq -\ln q$. For this,
     note that $h_q(-\ln q)=0$ and that
 \begin{align*}
    h'_q(t)=1-2a_qt-e^{-t}, \ h''_q(t)=-2a_q+e^{-t}.
 \end{align*}
     We have
 \begin{align*}
    \frac {(\ln q)^2h''_q(-\ln q)}{2}=\frac {q(\ln q)^2}{2}+\ln
    q+1-q,
 \end{align*}
    and the right-hand side expression above
    is easily seen to be $<0$ for $0<q<1$. As $h^{(3)}_q(t)<0$ for $t \geq -\ln q$,
    we conclude that $h''_q(t)<0$ for $t \geq -\ln q$. It's also
    easy to see that $h'_q(-\ln q)<0$ and we then deduce that $h'_q(t)<0$ for $t \geq -\ln
    q$ and this implies $h_q(t) \leq 0$ for $t \geq -\ln q$, which
    completes the proof of the ``if" part of the assertion of the
    theorem.

     For the ``only
     if" part of the assertion of the theorem, note that we have
     by \eqref{2.7} and \eqref{2.7'},
 \begin{align*}
    g_{q,c}(x+1)-g_{q,c}(x)=-\frac {(\ln q)q^x}{1-q^x}-\ln\frac
    {1-q^{x+1}}{1-q^{x}}-a_q\frac {(\ln
    q)^2q^{x+c}}{(1-q^{x+c})^2}.
 \end{align*}
    If we set $z=q^x$ and consider the Taylor expansion of the
    above expression at $z=0$, then the first order term is:
 \begin{align*}
    (-\ln q+q-1-a_q(\ln q)^2q^c)z=(h_q(-\ln q)+a_q(\ln q)^2-a_q(\ln
    q)^2q^c)t>0,
 \end{align*}
    if $c>0$. This implies that $g_{q,c}(x+1)>g_{q,c}(x)$ when $x$
    is large enough and this shows that $g_{q,c}(x)$ can't be
    completely monotonic on $(0, +\infty)$ when $c>0$ and this
    completes the proof of the ``only if" part of the assertion of
    the theorem.
 \end{proof}

    Similar to Theorem \ref{thm3}, one can prove the following
    result, whose proof we leave to the reader.
\begin{theorem}
 \label{thm2} Let $0<q<1$ be fixed. Let $c \geq 0$. The function
 \begin{align*}
    x \mapsto \psi_q(x)-\ln\frac {1-q^x}{1-q}+\frac {1}{2}\psi'_q(x+c)
 \end{align*}
    is completely monotonic on $(0, +\infty)$ if $c =0 $ and its negative is completely monotonic
on $(0, +\infty)$ if $c
    \geq 1/3$.
\end{theorem}

   Related to the function given in \eqref{4.1}, we have the following
   $q$-analogue:
\begin{theorem}
\label{thm1}
  Let $0<q<1$ be fixed, the functions
\begin{eqnarray}
\label{3.0}
  && -\psi_q(x)+\ln \Big ( \frac {1-q^x}{1-q} \Big )+ \frac {(\ln q) q^x }{2(1-q^x)}+\frac 1{12}\psi''_q(x+1/2),\\
\label{3.0'}
  && \psi_q(x)-\ln \Big ( \frac {1-q^x}{1-q} \Big )-\frac {(\ln q) q^x }{2(1-q^x)}-\frac
  1{12}\psi''_q(x)
\end{eqnarray}
  are completely monotonic on $(0, +\infty)$.
\end{theorem}
\begin{proof}
  The function given in \eqref{3.0} being completely monotonic on $(0, +\infty)$ follows from \eqref{2.5} and \eqref{1.12}. As by \eqref{1.12}, we have
\begin{equation*}
  \psi_q(x)-\ln \Big ( \frac {1-q^x}{1-q} \Big )- \frac {(\ln q) q^x }{2(1-q^x)}-\frac 1{12}\psi''_q(x+1/2) \\
  = \sum^{\infty}_{n=1}\Big (\frac {\ln q}{1-q^n}+\frac 1{n}-\frac {\ln q}{2}-\frac {n^2(\ln q)^3q^{n/2}}{12(1-q^n)} \Big )q^{nx}.
\end{equation*}
  Similarly, the function given in \eqref{3.0'} being completely monotonic on $(0, +\infty)$ follows from \eqref{2.6} and \eqref{1.12}.
\end{proof}

   Our next result is motivated by \eqref{1.1} and \eqref{1.2}:
\begin{theorem}
\label{thm1.0}
  Let $0<q<1$ be fixed, the functions
\begin{eqnarray}
\label{3.1}
&& \psi'_q(x)-\frac {(\ln q)^2q^x}{(1-q)(1-q^x)}-\frac {(\ln q)^2q^{2x}}{(1+q)(1-q^x)^2}, \\
\label{3.2} &&  -\psi'_q(x+1/2)+\frac {(\ln
q)^2q^{x+1/2}}{(1-q)(1-q^x)}
\end{eqnarray}
  are completely monotonic on $(0, +\infty)$.
\end{theorem}
\begin{proof}
  To show the function given in \eqref{3.1} is completely monotonic on $(0, +\infty)$, we note that
\begin{equation*}
  \frac {q^x}{(1-q^x)^2}=\sum^{\infty}_{n=1}nq^{nx}.
\end{equation*}
    Using this and \eqref{1.12}, we can recast \eqref{3.1} as
\begin{eqnarray*}
  && \psi'_q(x)-\frac {(\ln q)^2q^x}{(1-q)(1-q^x)}-\frac {(\ln q)^2q^{2x}}{(1+q)(1-q^x)^2} \\
&=& (\ln q)^2 \Big ( \sum^{\infty}_{n=1}\frac {nq^{nx}}{1-q^n}-
\frac {1}{1-q}\sum^{\infty}_{n=1}q^{nx}
- \frac {q^x}{1+q}\sum^{\infty}_{n=1}nq^{nx} \Big ) \\
  &=& (\ln q)^2 \sum^{\infty}_{n=2} \Big (\frac {n}{1-q^n}-\frac {1}{1-q}-\frac {n-1}{1+q} \Big )q^{nx}=(\ln q)^2 \sum^{\infty}_{n=2} \Big (\frac {u_{n}(q)}{(1-q^n)(1-q)(1+q)} \Big )q^{nx},
\end{eqnarray*}
  where
\begin{equation*}
   u_n(q)=n(1-q^2)-(1+q)(1-q^n)-(n-1)(1-q)(1-q^n)=n(1-q)(q+q^{n})-2q(1-q^n).
\end{equation*}
  It suffices to show that $u_n(q) \geq 0$ for $n \geq 2$, $0<q<1$, or equivalently,
\begin{equation*}
  n(1+q^{n-1}) \geq 2 \frac {1-q^n}{1-q}=2\sum^{n-1}_{i=0}q^i.
\end{equation*}
 It is easy to see that the function $q \mapsto 2\sum^{n-1}_{i=0}q^i-nq^{n-1}$ is an increasing function of $0<q \leq 1$
 and on considering the value of this function at $q=1$, we see that it implies $u_n(q) \geq 0$ for $n \geq 2$, $0<q<1$
 and this establishes our assertion on the function given in \eqref{3.1}.

  To show the function given in \eqref{3.2} is completely monotonic on $(0, +\infty)$, we use \eqref{1.12} to get
\begin{equation*}
   \psi'_q(x+1/2)-\frac {(\ln q)^2q^{x+1/2}}{(1-q)(1-q^x)}
  = (\ln q)^2 \sum^{\infty}_{n=1} \Big (\frac {nq^{n/2}}{1-q^n}-\frac {q^{1/2}}{1-q} \Big )q^{nx}.
\end{equation*}
   It suffices to show that $nq^{n/2-1/2} \leq (1-q^n)/(1-q)=\sum^{n-1}_{i=0}q^i$ for $0<q<1$. This follows by noting that $2\sum^{n-1}_{i=0}q^i=\sum^{n-1}_{i=0}(q^i+q^{n-i-1})$ and that $q^i+q^{n-i-1} \geq 2q^{n/2-1/2}$ by the arithmetic-geometric inequality and this completes the proof.
\end{proof}

\begin{cor}
\label{cor3.2} Let $0<q<1$ be fixed, then for $x>0$, we have
\begin{align*}
   \psi'_q(x+1/2) \leq \frac {(\ln
q)^2q^{x+1/2}}{(1-q)(1-q^x)}.
\end{align*}
\end{cor}
  The above inequality follows readily from Theorem \ref{thm1.0}
  on considering the value of the function given in \eqref{3.2} as
  $x \rightarrow +\infty$. As it's easy to see that $-(\ln
  q)q^{1/2} < 1-q$ when $0<q<1$, the above inequality gives a
  refinement of inequality \eqref{1.1}.


\end{document}